\documentclass[11pt, a4paper,leqno]{amsart}
\usepackage{amsmath,amsthm,amscd,amssymb,amsfonts, amsbsy}
\usepackage{latexsym}
\usepackage{txfonts}
\usepackage{exscale}
\usepackage{bbm}
\usepackage{enumitem}
\usepackage[colorlinks,citecolor=red,pagebackref,hypertexnames=false]{hyperref}
\usepackage{color}

\calclayout
\allowdisplaybreaks

\theoremstyle{plain}
\newtheorem{theorem}[equation]{Theorem}
\newtheorem{lemma}[equation]{Lemma}

\newtheorem{proposition}[equation]{Proposition}

\theoremstyle{definition}
\newtheorem{definition}[equation]{Definition}

\theoremstyle{remark}
\newtheorem{remark}[equation]{Remark}

\newtheorem{claim}[equation]{Claim}

\numberwithin{equation}{section}

\newcommand{\ZZ}{{\mathbb{Z}}}
\newcommand{\eps}{\varepsilon}

\newcommand{\dist}{\operatorname{dist}}

\newcommand{\dv}{\operatorname{div}}

\newcommand{\re}{\mathbb{R}}
\newcommand{\rn}{\mathbb{R}^n}

\newcommand{\ree}{\mathbb{R}^{n+1}}
\newcommand{\N}{\mathbb{N}}
\newcommand{\dd}{\mathbb{D}}

\newcommand{\om}{\Omega}

\newcommand{\F}{\mathcal{F}}
\newcommand{\cL}{\mathcal{L}}

\newcommand{\E}{\mathcal{E}}

\newcommand{\nn}{\mathcal{N}}

\newcommand{\W}{\mathcal{W}}
\newcommand{\B}{\mathcal{B}}

\newcommand{\sbf}{{\bf S}}

\newcommand{\G}{\mathcal{G}}

\newcommand{\nnt}{\widetilde{\mathcal{N}}}

\newcommand{\bt}{\widetilde{B}}

\newcommand{\pom}{\partial\Omega}

\newcommand{\hm}{\omega}
\newcommand{\vp}{\varphi}

\newcommand{\tuq}{\widetilde{U}_Q}
\newcommand{\tuqt}{\widetilde{U}_{Q,2\tau}}

\newcommand{\RNum}[1]{\uppercase\expandafter{\romannumeral #1\relax}}

\renewcommand{\emptyset}{\mbox{\textup{\O}}}

\DeclareMathOperator{\diam}{diam}

\DeclareMathOperator{\interior}{int}

\def\div{\mathop{\operatorname{div}}\nolimits}

\begin{document}
\allowdisplaybreaks

\title[Quantitative Fatou Theorems and Uniform Rectifiability]{Quantitative Fatou Theorems and Uniform Rectifiability}
\author{Simon Bortz}

\address{Simon Bortz \\
	School of Mathematics, University of Minnesota, Minneapolis, MN
55455, USA}
	\email{bortz010@umn.edu} 

\author{Steve Hofmann}

\address{Steve Hofmann
\\
Department of Mathematics
\\
University of Missouri
\\
Columbia, MO 65211, USA} \email{hofmanns@missouri.edu}

\thanks{This material is based upon work supported by National Science Foundation under Grant No.\ DMS-1440140 while the  authors were in residence at the MSRI in Berkeley, California, during the Spring 2017 semester. The first author was supported by the NSF INSPIRE Award DMS-1344235. The second author was supported by NSF grant DMS-1664047}

\maketitle

\date{\today}


\begin{abstract} We show that a suitable quantitative Fatou Theorem characterizes uniform rectifiability
in the codimension 1 case. 
\end{abstract}

\tableofcontents

\section{Introduction} 
 
Fatou theorems take their name from a classical result of Fatou \cite{F} concerning a.e. existence of boundary limits of bounded harmonic functions (see also \cite[Chapter 7]{St}).
In \cite{G}, Garnett proved a ``quantitative Fatou theorem" (\cite[Corollary 6.7]{G}) for bounded harmonic functions in the upper half-plane, which, roughly speaking, means the following:
given a bounded harmonic function $u$, normalized so that $\|u\|_\infty\leq 1$, and a number $\eps>0$, 
one counts, locally at each scale, and at each boundary point $x$, 
the maximum
number of oscillations of $u$, of size at least 
$\eps$, on any lacunary vertical (or non-tangential) sequence approaching $x$;  the resulting counting function
then enjoys an estimate of Carleson measure type, with bound depending only on $\eps$ and the parameter of lacunarity
(and the aperture of the non-tangential approach region).
Garnett's theorem was a corollary of the fact that bounded harmonic functions in the upper half-plane are $\epsilon$-approximable, a property first established by Varopoulos \cite{V}, and 
refined by Garnett \cite{G}; subsequently, the $\epsilon$-approximablity 
of bounded harmonic functions in Lipschitz domains in $\ree$ was obtained by Dahlberg \cite{D}.   
In \cite{KKPT}, the authors consider the case of a
(real) divergence form uniformly elliptic 
operator $\cL = -\dv A\nabla$ in a Lipschitz domain $\Omega$, and generalize Garnett's result by 
showing that quantitative Fatou theorems hold  
for any such $\cL$ whose bounded null solutions are $\epsilon$-approximable.  In turn, they then 
deduce that
elliptic-harmonic measure for $\cL$ belongs to the Muckenhoupt $A_\infty$ class with respect to surface measure
on $\pom$.  The latter implication is not available in settings as general as those 
we consider here, since
absolute continuity of harmonic measure with respect to surface measure may 
fail in the absence of sufficient connectivity, 
even for an open set with a uniformly rectifiable\footnote{See Definition \ref{defur}  below.} 
boundary \cite{BiJo}. 

More recently, in the current context, it was shown that if $\cL$ is a (real) divergence form operator satisfying the Carleson measure condition \eqref{admisop.eq} and the pointwise local Lipschitz bound 
\eqref{admisop2.eq}, then $\epsilon$-approximability of bounded null solutions to $\cL$ and $\cL^*$ is equivalent to uniform rectifiability; see \cite{HMM} and \cite{AGMT}. This is perhaps surprising, in light of the example of \cite{BiJo}; however, one may wonder what other surrogates, for the $A_\infty$ property of harmonic measure, do hold. In particular, the works of \cite{G} and \cite{KKPT} prompt two natural questions:  1) What is the appropriate notion of a quantitative Fatou theorem in an open set without traditional (connected)
accessibility? 2) Does this notion serve to characterize uniform rectifiability? The present work addresses these questions.
Our main result is the following.
\begin{theorem}\label{tmain}
Let $\Omega\subset \ree$ be an open set satisfying an interior Corkscrew condition (Definition \ref{def1.cork} below),
whose boundary is $n$-dimensional Ahlfors-David regular (Definition \ref{defadr}).  
Suppose that $\cL=-\dv A\nabla$ is a uniformly elliptic divergence form operator
whose coefficients satisfy \eqref{admisop.eq} and  \eqref{admisop2.eq}.
Then a quantitative Fatou theorem
holds for bounded null solutions of $\cL$ and its adjoint $\cL^*$, if and only if $\pom$ is uniformly rectifiable.
\end{theorem}

In the sequel, we shall explain precisely the meaning of the term ``quantitative Fatou theorem". 
For the moment, however, let us just say that this entails, as in Garnett's theorem, 
obtaining estimates of Carleson measure type for the counting function which gives bounds
(locally) on the $\eps$-oscillations of a bounded solution, on lacunary sequences lying along a
non-tangential, but {\it possibly disconnected} path to the boundary.   In our context, 
the lack of connectivity cannot be avoided (indeed, there may be no connected non-tangential path), 
and is a rather delicate matter.
In particular, there may be multiple (even infinitely many)
choices of non-tangential approach to the boundary, and some (perhaps most) of these may not work
(heuristically, while the good paths may be disconnected, they cannot jump around too much).  Instead,
there are canonical,
universally defined non-tangential approach regions, which may be localized to define an appropriate counting function.
In the sequel, we shall find it convenient to construct these approach regions 
dyadically.

We have decoupled the two parts of Theorem \ref{tmain}, with precise statements, into Theorem \ref{QFimpUR.thrm}
(quantitative Fatou implies uniform rectifiability), and Theorem \ref{dydfatthrm.thrm} (uniform rectifiability implies
quantitative Fatou).   We remark that we shall obtain
Theorem \ref{QFimpUR.thrm} as an essentially immediate corollary of Theorem \ref{QFimpCorona.thrm}, using the results of \cite{AGMT}.

\section{Preliminaries}

Throughout we will assume $n \ge 2$.
\begin{definition}[\bf  ADR]\label{defadr}  (aka {\it Ahlfors-David regular}).
We say that a  set $E \subset \ree$, of Hausdorff dimension $n$, is ADR
if it is closed, and if there is some uniform constant $C$ such that
\begin{equation} \label{eq1.ADR}
\frac1C\, r^n \leq \sigma\big(\Delta(x,r)\big)
\leq C\, r^n,\quad\forall r\in(0,\diam (E)),\ x \in E,
\end{equation}
where $\diam(E)$ may be infinite.
Here, $\Delta(x,r):= E\cap B(x,r)$ is the ``surface ball" of radius $r$,
and $\sigma:= H^n|_E$ 
is the ``surface measure" on $E$, where $H^n$ denotes $n$-dimensional
Hausdorff measure.
\end{definition}

\begin{definition} ({\bf Corkscrew condition}).  \label{def1.cork}
Following
\cite{JK}, we say that an open set $\Omega\subset \ree$
satisfies the (interior) ``Corkscrew condition'' if for some uniform constant $c>0$ and
for every surface ball $\Delta:=\Delta(x,r),$ with $x\in \partial\Omega$ and
$0<r<\diam(\partial\Omega)$, there is a ball
$B(X_\Delta,cr)\subset B(x,r)\cap\Omega$.  The point $X_\Delta\subset \Omega$ is called
a ``Corkscrew point'' relative to $\Delta.$  We note that  we may allow
$r<C\diam(\pom)$ for any fixed $C$, simply by adjusting the constant $c$.
\end{definition}

Henceforth we will assume that $\om \subset \ree$ is an open set satisfying the (interior) corkscrew condition such that $\pom$ is ADR.

\begin{definition}[\bf Divergence Form Elliptic Operator] We say that $\cL = -\div A \nabla$ is a divergence form elliptic operator if there exists $C > 1$ such that 
$$C^{-1}|\xi|^2 \le \langle A(X)\xi, \xi \rangle, \quad \lVert A \rVert_\infty \le C,$$
for all $\xi,X \in \ree$. We interpret the operator $\cL$ in the weak sense as usual. 
\end{definition}

We shall consider solutions to divergence form elliptic operators $\cL$ on open sets $\om$. Sometimes we will impose the additional assumption that the coefficients of $\cL$ are locally Lipschitz in $\Omega$, and 
satisfy the Carleson measure condition
\begin{equation}\label{admisop.eq}
\sup_{\substack{x \in \pom \\ 0 < r < \diam(\pom)}} \frac{1}{H^n(B(x,r) \cap \pom)} \iint_{B(x,r) \cap \om} |\nabla A(X)| \, dX \le C< \infty\,,
\end{equation}
as well as the pointwise gradient bound
\begin{equation}\label{admisop2.eq}
|\nabla A(X)|\leq C \delta(X)^{-1}\,.
\end{equation}

\begin{definition}\label{defur} ({\bf UR}) (aka {\it uniformly rectifiable}).
An $n$-dimensional ADR (hence closed) set $E\subset \ree$
is UR if and only if it contains ``Big Pieces of
Lipschitz Images" of $\rn$ (``BPLI").   This means that there are positive constants $\theta$ and
$M_0$, such that for each
$x\in E$ and each $r\in (0,\diam (E))$, there is a
Lipschitz mapping $\rho= \rho_{x,r}: \rn\to \ree$, with Lipschitz constant
no larger than $M_0$,
such that 
$$
H^n\Big(E\cap B(x,r)\cap  \rho\left(\{z\in\rn:|z|<r\}\right)\Big)\,\geq\,\theta\, r^n\,.
$$
\end{definition}

\begin{definition}\label{defurchar} ({\bf ``UR character"}).   Given a UR set $E\subset \ree$, its ``UR character"
is just the pair of constants $(\theta,M_0)$ involved in the definition of uniform rectifiability,
along with the ADR constant; or equivalently,
the quantitative bounds involved in any particular characterization of uniform rectifiability.
\end{definition}

We employ the following standard notation:
\begin{list}{$\bullet$}{\leftmargin=0.4cm  \itemsep=0.2cm}

\item We use the letters $c,C$ to denote harmless positive constants, not necessarily
the same at each occurrence, which depend only on dimension and the
constants appearing in the hypotheses of the theorems (which we refer to as the
``allowable parameters'').  We shall also
sometimes write $a\lesssim b$ and $a \approx b$ to mean, respectively,
that $a \leq C b$ and $0< c \leq a/b\leq C$, where the constants $c$ and $C$ are as above, unless
explicitly noted to the contrary.

\item Given a closed set $E \subset \ree$, we shall
use lower case letters $x,y,z$, etc., to denote points on $E$, and capital letters
$X,Y,Z$, etc., to denote generic points in $\ree$ (especially those in $\ree\setminus E$).

\item The open $(n+1)$-dimensional Euclidean ball of radius $r$ will be denoted
$B(x,r)$ when the center $x$ lies on $E$, or $B(X,r)$ when the center
$X \in \ree\setminus E$.  A ``surface ball'' is denoted
$\Delta(x,r):= B(x,r) \cap\partial\Omega.$

\item Given a Euclidean or surface ball $B= B(X,r)$ or $\Delta = \Delta(x,r)$, its concentric
dilate by a factor of $\kappa >0$ will be denoted
$\kappa B := B(X,\kappa r)$ or $\kappa \Delta := \Delta(x,\kappa r).$

\item Given a Euclidean ball $B$ (resp., a surface ball $\Delta$), we shall denote its radius by
$r_B$ (resp. $r_\Delta$).

\item Given a (fixed) closed set $E \subset \ree$, for $X \in \ree$, we set $\delta(X):= \dist(X,E)$. If we are working with an open set, $\om$, we will use the notation $\delta(X) := \dist(X,\pom)$, that is, we will take $E = \pom$. 

\item We let $H^n$ denote $n$-dimensional Hausdorff measure, and let
$\sigma := H^n\big|_{E}$ denote the ``surface measure'' on a closed set $E$
of co-dimension 1.

\item For a Borel set $A\subset \ree$, we let $\mathbbm{1}_A$ denote the usual
indicator function of $A$, i.e. $\mathbbm{1}_A(x) = 1$ if $x\in A$, and $\mathbbm{1}_A(x)= 0$ if $x\notin A$.

\item For a Borel set $A\subset \ree$,  we let $\interior(A)$ denote the interior of $A$.


\item We shall use the letter $J$
to denote a closed $(n+1)$-dimensional Euclidean dyadic cube with sides
parallel to the co-ordinate axes, and we let $\ell(J)$ denote the side length of $J$.
Given an ADR set $E\subset \ree$, we use $Q$ to denote a dyadic ``cube''
on $E$.  The
latter exist (cf. \cite{DS1}, \cite{Ch}), and enjoy certain properties
which we enumerate in Lemma \ref{lemmaCh} below.

\end{list}

\begin{definition}({\bf Harnack Chain condition}).  \label{def1.hc} Following \cite{JK}, we say
that $\Omega$ satisfies the Harnack Chain condition if there is a uniform constant $C$ such that
for every $\rho >0,\, \Lambda\geq 1$, and every pair of points
$X,X' \in \Omega$ with $\delta(X),\,\delta(X') \geq\rho$ and $|X-X'|<\Lambda\,\rho$, there is a chain of
open balls
$B_1,\dots,B_N \subset \Omega$, $N\leq C(\Lambda)$,
with $X\in B_1,\, X'\in B_N,$ $B_k\cap B_{k+1}\neq \emptyset$
and $C^{-1}\diam (B_k) \leq \dist (B_k,\partial\Omega)\leq C\diam (B_k).$  The chain of balls is called
a ``Harnack Chain''.
\end{definition}

\begin{definition}({\bf NTA}). \label{def1.nta} Again following \cite{JK}, we say that a
domain $\Omega\subset \ree$ is NTA ({\it Non-tangentially accessible}) if it satisfies the
Harnack Chain condition, and if both $\Omega$ and
$\Omega_{\rm ext}:= \ree\setminus \overline{\Omega}$ satisfy the Corkscrew condition.
\end{definition}

\begin{definition}({\bf CAD}). \label{def1.cad}  We say that a connected open set $\om \subset \ree$
is a CAD ({\it Chord-arc domain}), if it is NTA, and if $\pom$ is ADR.
\end{definition}

\begin{definition}[\bf DeG/N/M Estimates] Given an elliptic operator, $\cL = -\div A \nabla$, we say that solutions to $\cL u = 0$ on $\om$ satisfy De Giorgi-Nash-Moser (DeG/N/M)  estimates there exist $C, \beta > 0$ if for every ball $B = B(x,r)$ such that $2B = B(x,2r) \subset \om$ we have
$$|u(Y) - u(X)| \le C \left(\frac{|X - Y|}{r}\right)^\beta \left(\fint_{2B}|u(Z)|^2 \, dZ \right)^\frac{1}{2},$$
whenever $X, Y \in B$ (see \cite{DeG,N}). We note that {\it all} operators with real coefficients satisfy DeG/N/M estimates. 
\end{definition}

\begin{definition}[\bf $\epsilon-$approximablity]\label{epsapprox.def} Let $\om \subset \ree$ be an open set satisfying the interior corkscrew condition with ADR boundary and let $\epsilon \in (0,1)$. We say that $u$, with $\lVert u \rVert_{L^\infty{(\om})} \le 1$ is $\epsilon-$approximable, if there is a constant $C_\epsilon$ and a function $\varphi = \varphi_\epsilon \in W^{1,1}_{loc}(\om)$ satisfying
$$\lVert u - \varphi \rVert_{L^\infty(\om)} < \epsilon$$
and 
\begin{equation}\label{Cestepsapprox.eq}
\sup_{x \in E, 0 < r < \infty} \frac{1}{r^n} \int_{B(x,r) \cap \om} |\nabla \varphi(Y)| \, dY \le C_\epsilon.
\end{equation}
Given $\epsilon \in (0, 1)$ we say every bounded solution of $\cL u = 0$ is $\epsilon$-approximable if for all $u$ with $\cL u = 0$ and $\lVert u \rVert_{L^\infty(\om)} \le 1$, $u$ is $\epsilon$-approximable and the constant $C_\epsilon$ is independent of $u$.
\end{definition}

\begin{theorem}[\cite{HMM}]\label{epsapproxthrm.thrm} Suppose $\om \subset \ree$ is an open set satisfying the (interior) corkscrew condition such that $\pom$ is UR and $\cL$ is a divergence form elliptic operator with coefficients satisfying \eqref{admisop.eq} and \eqref{admisop2.eq}. Then bounded solutions to $\cL u= 0$ in $\om$ are $\epsilon$-approximable for all $\epsilon \in (0,1)$ with constant $C_\epsilon$ depending on \eqref{admisop.eq}, \eqref{admisop2.eq}, $\epsilon$, $n$ and the $UR$ character of $\pom$.
\end{theorem}

\begin{remark}\label{remark2.16}
In fact, this result is proved explicitly in \cite{HMM} only in the case that $\cL$ is the Laplacian, but as noted in 
\cite[Remark 5.29]{HMM}, the proof in fact does not require harmonicity of
$u$, {\it per se}, but only the following properties of $u$: 1) $u\in L^\infty(\Omega)$, with $\|u\|_\infty\leq 1$; 2) $u$ satisfies Moser's local boundedness estimates in $\Omega$; 3) $u$ satisfies the Carleson measure estimate 
\begin{equation}
\sup_{x\in E,\, 0<r<\infty} \,\frac1{r^n}\iint_{B(x,r)} |\nabla u(Y)|^2 \delta(Y) \,dY\,
\leq \,C\, \|u\|^2_{L^\infty(\Omega)}\, ,
\end{equation}
and 4) $u$ satisfies ``$N<S$" estimates\footnote{I.e., that the non-tangential maximal function of $u$
is controlled in some $L^p$ norm by the conical square function of $\nabla u$.} in chord-arc subdomains of $\Omega$, with
uniform quantitative bounds depending on the chord-arc constants
(we mention here that item 4) was inadvertently omitted in \cite[Remark 5.29]{HMM}).
We further remark that these ingredients are all in place, for $u$ as in Theorem \ref{epsapproxthrm.thrm}, even with 
\eqref{admisop.eq} replaced by the weaker condition
\begin{equation}\label{eq2.18}
\sup_{\substack{x \in \pom \\ 0 < r < \diam(\pom)}} \frac{1}{H^n(B(x,r) \cap \pom)} 
\iint_{B(x,r) \cap \om} |\nabla A(X)|^2 \, \delta(X) \,dX \le C< \infty\,.
\end{equation}
\end{remark}

\begin{lemma}\label{lemmaCh}\textup{({\bf Existence and properties of the ``dyadic grid''})
\cite{DS1,DS2}, \cite{Ch}.}
Suppose that $E\subset \ree$ is closed $n$-dimensional ADR set.  Then there exist
constants $ a_0>0$ and $C_1<\infty$, depending only on dimension and the
ADR constant, such that for each $k \in \mathbb{Z},$
there is a collection of Borel sets (``cubes'')
$$
\mathbb{D}_k:=\{Q_{j}^k\subset E: j\in \mathfrak{I}_k\},$$ where
$\mathfrak{I}_k$ denotes some (possibly finite) index set depending on $k$, satisfying

\begin{list}{$(\theenumi)$}{\usecounter{enumi}\leftmargin=.8cm
\labelwidth=.8cm\itemsep=0.2cm\topsep=.1cm
\renewcommand{\theenumi}{\roman{enumi}}}

\item $E=\cup_{j}Q_{j}^k\,\,$ for each
$k\in{\mathbb Z}$.

\item If $m\geq k$ then either $Q_{i}^{m}\subset Q_{j}^{k}$ or
$Q_{i}^{m}\cap Q_{j}^{k}=\emptyset$.

\item For each $(j,k)$ and each $m<k$, there is a unique
$i$ such that $Q_{j}^k\subset Q_{i}^m$.

\item $\diam\big(Q_{j}^k\big)\leq 2^{-k}$.

\item Each $Q_{j}^k$ contains some ``surface ball'' $\Delta \big(x^k_{j},a_02^{-k}\big):=
B\big(x^k_{j},a_02^{-k}\big)\cap E$.

\end{list}
\end{lemma}

A few remarks are in order concerning this lemma.

\begin{list}{$\bullet$}{\leftmargin=0.4cm  \itemsep=0.2cm}

\item In the setting of a general space of homogeneous type, this lemma has been proved by Christ
\cite{Ch}, with the
dyadic parameter $1/2$ replaced by some constant $\delta \in (0,1)$.
In fact, one may always take $\delta = 1/2$ (cf.  \cite[Proof of Proposition 2.12]{HMMM}).
In the presence of the Ahlfors-David
property (\ref{eq1.ADR}), the result already appears in \cite{DS1,DS2}.

\item  For our purposes, we may ignore those
$k\in \mathbb{Z}$ such that $2^{-k} \gtrsim {\rm diam}(E)$, in the case that the latter is finite.

\item  We shall denote by  $\mathbb{D}=\mathbb{D}(E)$ the collection of all relevant
$Q^k_j$, i.e., $$\mathbb{D} := \cup_{k} \mathbb{D}_k,$$
where, if $\diam (E)$ is finite, the union runs
over those $k$ such that $2^{-k} \lesssim  {\rm diam}(E)$.

\item Properties $(iv)$ and $(v)$ imply that for each cube $Q\in\mathbb{D}_k$,
there is a point $x_Q\in E$, a Euclidean ball $B(x_Q,r)$ and a surface ball
$\Delta(x_Q,r):= B(x_Q,r)\cap E$ such that
$r\approx 2^{-k} \approx {\rm diam}(Q)$
and \begin{equation}\label{cube-ball}
\Delta(x_Q,r)\subset Q \subset \Delta(x_Q,Cr),\end{equation}
for some uniform constant $C$.
We shall denote this ball and surface ball by
\begin{equation}\label{cube-ball2}
B_Q:= B(x_Q,r) \,,\qquad\Delta_Q:= \Delta(x_Q,r),\end{equation}
and we shall refer to the point $x_Q$ as the ``center'' of $Q$.

\item For a dyadic cube $Q\in \mathbb{D}_k$, we shall
set $\ell(Q) = 2^{-k}$, and we shall refer to this quantity as the ``length''
of $Q$.  Evidently, $\ell(Q)\approx \diam(Q).$

\item For any $\lambda > 1$ and $Q \in \dd(E)$ we will write 
\begin{equation}\label{lambdaQ.eq}
\lambda Q = \{x \in E: \dist(x,Q) \le (\lambda - 1) \ell(Q)\}.
\end{equation}


\end{list}

Later, we will consider stopping time regimes, making the following definition useful.

\begin{definition}\cite{DS2}.\label{d3.11}   
Let $\sbf\subset \dd(E)$. We say that $\sbf$ is
``coherent" if the following conditions hold:
\begin{itemize}\itemsep=0.1cm

\item[$(a)$] $\sbf$ contains a unique maximal element $Q(\sbf)$, which contains all other elements of $\sbf$ as subsets.

\item[$(b)$] If $Q$  belongs to $\sbf$, and if $Q\subset \widetilde{Q}\subset Q(\sbf)$, then $\widetilde{Q}\in {\bf S}$.

\item[$(c)$] Given a cube $Q\in \sbf$, either all of its children belong to $\sbf$, or none of them do.

\end{itemize}
We say that $\sbf$ is ``semi-coherent'' if only conditions $(a)$ and $(b)$ hold. 
\end{definition}

Given an open set $\om \subset \ree$ satisfying the (interior) corkscrew condition such that $\pom$ is ADR, we let $\W = \{J\}$ be a Whitney decomposition of $\om$, that is, $\{J\}$ is a collection of closed $(n+1)$-dimensional cubes whose interiors are disjoint, union is $\om$, for which 
\begin{equation}\label{whitprop.eq}
4 \diam(J)\leq
\dist(4J,\pom)\leq \dist(J,\pom) \leq 40\diam(J)\,,\qquad \forall\, J\in \mathcal{W}
\end{equation}
and 
$$1/4 \diam(J_1) \le \diam(J_2) \le 4 \diam(J_1)$$
whenever $J_1 \cap J_2 \neq \emptyset$. Given $\eta < 1$ and $K > 1$ we define for every $Q \in \dd(\pom)$,

\begin{equation}\label{eq3.1}
W_Q^0 = \{J \in \W: \eta^{1/4} \ell(Q) \le \ell(J) \le K^{1/2}\ell(Q), \dist(J,Q) \le K^{1/2} \ell(Q)\}\,.
\end{equation}
When it seems useful to emphasize the dependence on $\eta$ and $K$,
 we shall write $\W_Q^0(\eta,K)$ in place of $W_Q^0$. 




\begin{remark}\label{remark:E-cks} We note that for an open 
set $\Omega$ satisfying the Corkscrew condition,
 $\W_Q^0=\W^0_Q(\eta,K)$ is non-empty,
 provided that we choose $\eta$ small enough, and $K$ large enough, depending only on 
the Corkscrew constant.  
In the sequel, we shall always
assume that $\eta$ and $K$ have been so chosen.
 \end{remark}
 
  We fix a small parameter $\tau_0>0$ (depending on dimension), so that
for any $J\in \W$, and any $\tau \in (0,\tau_0]$,
the concentric dilate
\begin{equation}\label{whitney1}
J^*(\tau):= (1+\tau) J
\end{equation} 
still satisfies the Whitney property
\begin{equation}\label{whitney}
\diam J\approx \diam J^*(\tau) \approx \dist\left(J^*(\tau), \pom\right) \approx \dist(J,\pom)\,, \quad 0<\tau\leq \tau_0\,.
\end{equation}
Moreover,
for $\tau\leq\tau_0$, with $\tau_0$ small enough, and for any $J_1,J_2\in \W$,
we have that $J_1^*(\tau)$ meets $J_2^*(\tau)$ if and only if
$J_1$ and $J_2$ have a boundary point in common, and that, if $J_1\neq J_2$,
then $J_1^*(\tau)$ misses $(3/4)J_2$.

We then define for all $\tau \in (0, \tau_0/2]$
\begin{equation}\label{eq.UQdef}
U_Q = U_Q(\eta, K, \tau) := \bigcup_{W_Q^0(\eta, K)} \interior( J^*(\tau)\,).
\end{equation} 
Note that our $U_Q$ 
is somewhat different to 
the constructions in \cite{HM-UR1, HMM} (we shall recall the latter constructions in Section \ref{s4}).
In the sequel, we will often suppress the dependence on $\eta, K$ and $\tau$ when these parameters have been fixed,
in order to simplify the notation.

Let us remark that for any fixed $\eta$ and $K$ there exists $N = N(\eta, K)$ such that 
$\#\{J \in W_Q^0(\eta, K)\} \le N$. It follows that $U_Q$ has only finitely many connected components, we will enumerate these connected components as $\{U_Q^i\}_{I_Q}$, where we have $\#I_Q < N$.

\begin{definition}[Index Catalog and Subcatalogs]Let $\om \subset \ree$ be an open set satisfying the (interior) corkscrew condition such that $\pom$ is ADR. Given $\eta \ll 1 \ll K$ we enumerate the components of each $U_Q$ as $\{U_Q^i\}_{i \in I_Q}$ as before and call $\mathcal{I} = \{I_Q\}_{Q \in \dd(\pom)}$ the index catalog. We say that $I$ is an index subcatalog (or just a subcatalog) if $I = \{i_Q\}_{Q \in \dd(\pom)}$, where for each $Q$, $i_Q \in I_Q$; i.e., in a subcatalog,
we have fixed precisely one component of $U_Q$
for each $Q$.

\end{definition}

\begin{definition}[Admissible Sequences and the Dyadic Oscillation Counting Function] \label{def2.40}
Let $\om \subset \ree$ be an open 
set satisfying the (interior) corkscrew 
condition such that $\pom$ is ADR and $\eta \ll 1 \ll K$.  Let $u : \om \to \re$. Given a 
subcatalog $I$, a cube $Q \in \dd(\pom)$, a point $x\in Q$, and a 
number $\epsilon > 0$, we say that a sequence $\{X_k\}_{k=1}^{k_0+1}\subset \Omega$,
of arbitrary finite cardinality $k_0+1\geq 2$, 
is $(x,\epsilon,I,Q)$-{\it admissible for $u$} (or simply $(x,\epsilon,I,Q)$-admissible, when $u$ is 
understood from context)
if there exist strictly nested
cubes $\{Q_k\}_{k=1}^{k_0+1}$ with $x\in Q_{k_0+1}\subsetneq Q_{k_0}\subsetneq ...
\subsetneq Q_1  \subseteq Q$,
such that $X_k \in U_{Q_k}^{i_{Q_k}}$ with $i_{Q_k} \in I$, and $|u(X_k) -u(X_{k+1})|>\epsilon$. 

The dyadic oscillation counting function is then defined to be
\begin{equation}
\nn^{Q}u(x,\epsilon, I)
:= \sup\{k_0 : \exists \, (x,\epsilon,I,Q)\text{-admissible} \{X_k\}_{k=1}^{k_0+1}
\}\,.
\end{equation}
If there is no such $(x,\epsilon,I,Q)$-admissible sequence of cardinality at least 2, 
we set  $\nn^{Q}u(x,\epsilon, I)=0$.
\end{definition}

\begin{remark}
It is easy to see that for each $Q$,
$\nn^{Q}$ is $\sigma$-measurable. First, we 
may define a collection of intermediate functions $\nn^{Q}_j$ for $j \in \mathbb{Z}$, where we restrict the cubes $Q_k$ in the definition of $\nn$ to those which have side length greater than or equal to $2^{-j}$. This restriction yields a bounded simple function and taking the supremum over $j \in \ZZ$ is exactly $\nn^{Q}$.
\end{remark}

 Following \cite{AGMT}, we make the following definition.

\begin{definition}[Corona Decomposition for Harmonic Measure]\label{coronahm.def}
Let $\om \subset \ree$ be an open set satisfying the interior corkscrew condition with $n$-dimensional ADR boundary. Let $\cL$ be a (real) divergence form elliptic operator and $\hm_{\cL}$ be the corresponding elliptic measure for $\om$. We say that $\hm_{\cL}$ admits a Corona decomposition if $\dd(\pom)$ is decomposed into disjoint coherent stopping time 
regimes $\dd(\pom) = \bigcup_{\sbf'} \sbf'$ such that the following holds. The maximal cubes, $Q(\sbf')$, satisfy a Carleson 
packing condition
$$\sum_{Q(\sbf') \subset R} \sigma(Q(\sbf')) \le C \sigma(R), \quad \forall R \in \dd(\pom).$$
For each $Q(\sbf')$ there exists $p_{Q(\sbf')} \in \om$ such that 
$$c^{-1} \ell(Q(\sbf'))\le  \dist(p_{Q(\sbf')},Q(\sbf')) \le \dist(p_{Q(\sbf')},\pom) \le c \ell(Q(\sbf')),$$
so that
$$\hm_{\cL}^{p_{Q(\sbf')}}(3R) \approx \frac{\sigma(R)}{\sigma(Q(\sbf'))} \quad \forall R \in \sbf',$$
where the implicit constants and $c$ are uniform in $\sbf'$ and $R$. 
\end{definition}

\section{A Quantitative Fatou Theorem Implies Uniform Rectifiability}
The following Lemma is similar to Lemma 3.3 in \cite{AGMT}, with two differences. We do not obtain an estimate on the gradient of $u_Q$ or any approximant and we create a dichotomy which allows us to get a worse, but sufficient estimate for the purpose of packing low density cubes.

\begin{lemma}\label{constuq.lem} Let $\om\subset \ree$ be an open set with $n$-dimensional ADR boundary and $\cL$ a (real) divergence form elliptic operator.  Let $a \in (0,1)$ then there exists $\gamma \in (0,1/2)$ and $\epsilon' \in (0, 1/2)$ depending on $a, n$, ADR and the ellipticity constants such for every $\epsilon \in (0,\epsilon']$ the following holds. If $p_Q, s_Q \in \om$ are such that
$$a \epsilon \ell(Q) \le |p_Q - x_Q| \le \epsilon \ell(Q)$$
and 
$$|s_Q - x_Q| \le \gamma \epsilon \ell(Q)$$
and $E_Q \subseteq Q$ with
\begin{equation}\label{constuq.eq1}
\hm_{\cL}^{p_Q}(E_Q) \ge (1- \epsilon) \hm_{\cL}^{p_Q}(Q)
\end{equation}
then there exists a positive solution to $\mathcal{L} u = 0$,  $u_Q$, such that
\begin{itemize}
\item $u_Q(X) = \int f_Q \, \hm_{\cL}^X$, for a positive Borel function $f_Q$ satisfying $0\le f_Q \le \mathbbm{1}_{E_Q}$, 
\item $|u_Q(p_Q) - u_Q(s_Q)| \ge \frac{1}{2} \epsilon^\alpha,$
\end{itemize}
where $\alpha\in (0,1)$ depends on dimension, ADR and the ellipticity constant for $\cL$.
\end{lemma}
\begin{proof} For simplicity of notation we drop the subscript $\cL$ in $\hm_{\cL}$.
By a simple argument using Bourgain's Lemma \cite[Lemma 11.21]{HKM} (see also \cite[Lemma 3.2]{AGMT}), we have for all sufficiently small $\epsilon$ depending on $n$, ADR and ellipticity
\begin{equation}\label{constuq.eq2}
\hm^{p_Q}(Q) \ge (1 - c \epsilon^\alpha),
\end{equation}
where $\alpha \in(0,1)$ and $c > 0$ depend on dimension, ADR and ellipticity. Now we break into cases.
\\ {\bf Case 1:} $\hm^{s_Q}(E_Q) \le (1- \epsilon^\alpha)\hm^{p_Q}(E_Q)$.
In this case we set $u_Q(X) = \hm^X(E_Q)$. Using \eqref{constuq.eq1} and \eqref{constuq.eq2} we obtain
\begin{align*}
u_Q(p_Q) - u_Q(s_Q) &= \hm^{p_Q}(E_Q) - \hm^{s_Q}(E_Q)
\\ & \ge \hm^{p_Q}(E_Q)[ 1 - (1 - \epsilon^\alpha)]
\\ & \ge  \frac{1}{2} \epsilon^\alpha,
\end{align*}
provided $\epsilon \ll 1$ depending on $n$, ADR and ellipticity. As $u_Q$ obviously satisfies the other desired conditions the lemma is shown in this case.
\\ {\bf Case 2:} $\hm^{s_Q}(E_Q) \ge (1- \epsilon^\alpha)\hm^{p_Q}(E_Q)$. Using \eqref{constuq.eq2} we record the following
\begin{equation}\label{constuq.eq3}
\begin{split}
\hm^{s_Q}(\pom \setminus E_Q) &\le 1 - [(1 - \epsilon^\alpha)\hm^{p_Q}(E_Q)]
\\& \le 1 - [(1- \epsilon^\alpha)(1 -\epsilon)(1 -c\epsilon^\alpha)]
\\ &= (1 + c)\epsilon^\alpha - c\epsilon^{2\alpha} + \epsilon - (1 + c)\epsilon^{\alpha + 1} + c \epsilon^{2\alpha + 1}
\\ &\le \tilde{c}\epsilon^\alpha,
\end{split}
\end{equation}
provided that $\epsilon$ is small depending on dimension, ADR and ellipticity. Set $\Delta' := \Delta(x_Q, \gamma \epsilon \ell(Q)) = B(x_Q, \gamma \epsilon \ell(Q)) \cap \pom$. Define
$$\tilde{g}_Q(X) := \int_{\Delta'} = \frac{1}{\gamma \epsilon \ell(Q)} \E_{\mathcal{L}}(X,y)\, d\sigma(y),$$
where $\E_{\mathcal{L}}$ is the fundamental solution for $\mathcal{L}$. 
We note that 
$$0 \le \E_{\mathcal{L}}(X,Y) \approx \frac{1}{|X - Y|^{n -1}},$$
with implicit constants depending on dimension and ellipticity. Then by the ADR condition $\lVert \tilde{g}_Q\rVert_\infty \approx 1$, with implicit constants depending on dimension, ADR and ellipticity. Set 
$$g_Q(X) := \frac{1}{\lVert \tilde{g}_Q\rVert_\infty} \tilde{g}_Q(X).$$
A simple calculation shows that $g_Q(s_Q) \approx 1$ with constants independent of $\gamma$ and $\epsilon$. On the other hand, if $\gamma < \frac{a}{2}$ we have by the triangle inequality that
$$\frac{a}{2}\epsilon \ell(Q) \le |p_Q - y| \le \frac{3a}{2}\epsilon \ell(Q), \quad \forall y \in \Delta'.$$
Consequently, the ADR condition yields
\begin{align*}
g_Q(p_Q) &\approx \sigma(\Delta') \frac{1}{\gamma \epsilon\ell(Q)}\frac{1}{(\epsilon \ell(Q))^{n-1}}
\\& \approx \gamma^{n-1},
\end{align*}
where the implicit constants are independent of $\epsilon$. It follows for some $\gamma$ sufficiently small 
\begin{equation}\label{constuq.eq4}
|g_Q(p_Q) - g_Q(s_Q)| \gtrsim 1,
\end{equation}
where the choice of $\gamma$ and the implicit constant are independent of $\epsilon$. 
Having fixed $\gamma$, we set 
$$u_Q(X) := \int f_Q \, d\hm^X = \int_{E_Q} g_Q \hm^X.$$
Note that $f_Q$ has the desired property $0\le f_Q \le \mathbbm{1}_{E_Q}$.
Since $g_Q(X) = \int g_Q \, d\hm^X$, we have 
\begin{equation}\label{constuq.eq5}
|g_Q(X) - u_Q(X)| = \left|\int_{\pom \setminus E_Q} g_Q\,  \hm ^X\right| \le \hm^X(\pom\setminus E_Q).
\end{equation}
By our assumption that \eqref{constuq.eq1} holds and \eqref{constuq.eq2} we have
$$\hm^{p_Q}(\pom \setminus E_Q) \le 1 - (1 - \epsilon)(1 -c\epsilon^\alpha) \le \tilde{c}\epsilon^\alpha.$$
Combining this estimate and \eqref{constuq.eq3} with \eqref{constuq.eq5} we obtain the pair of estimates
$$|g_Q(p_Q) - u_Q(p_Q)|, |g_Q(s_Q) - u_Q(s_Q)| \le \tilde{c}\epsilon^\alpha.$$
Then for all $\epsilon$ sufficiently small depending on $\gamma$, \eqref{constuq.eq4} yields
$$|u_Q(p_Q) - u_Q(s_Q)| \gtrsim 1  \ge \frac{1}{2} \epsilon^\alpha.$$
The other properties of $u_Q$ are again easily checked in this case.
\end{proof}

The following is really the main result in this section.  
We recall that the dyadic counting function 
$ \nn^{Q_0}u(x,\epsilon,I)$ is defined in Definition \ref{def2.40}.

\begin{theorem}\label{QFimpCorona.thrm} Let $\om \subset \ree$ be an open set satisfying the corkscrew condition with $n$-dimensional ADR boundary, $\pom$, and let $\mathcal{L}$ be a (real) divergence form operator. Let $\eta \ll 1 \ll K$ be such that every Whitney region $U_Q$ is non-empty and $\tau \in (0,\tau_0/2]$. There exists $\epsilon_0 = \epsilon_0(n, ADR, \eta, K, \mathcal{L})$ such that the following holds. If there exists {\it any} subcatalog $I$ with the property that for any bounded solution to $\mathcal{L} u = 0$ in $\om$ with $\lVert u \rVert_\infty \le 1$
\begin{equation}\label{QFholdsforQFimpUR.eq}
\int_{Q_0} \nn^{Q_0}u(x,\epsilon_0,I) \lesssim \sigma(Q_0), \quad \forall Q_0 \in \dd(\pom),
\end{equation}
then $\hm_\mathcal{L}$ admits a  Corona decomposition (see Definition \ref{coronahm.def}). 
\end{theorem}
Combining Theorem \ref{QFimpCorona.thrm} with the proof of the main result of \cite{AGMT} we obtain the following
as an immediate corollary.
\begin{theorem}\label{QFimpUR.thrm} Let $\om \subset \ree$ be an open set satisfying the corkscrew condition with $n$-dimensional ADR boundary, $\pom$, and let $\mathcal{L}$ be a (real) divergence form operator with coefficients satisfying \eqref{admisop.eq} and \eqref{admisop2.eq}. Let $\eta \ll 1 \ll K$ be such that every Whitney region $U_Q$ is non-empty and $\tau \in (0,\tau_0/2]$. There exists $\epsilon_0 = \epsilon_0(n, ADR, \eta, K, \mathcal{L})$ such that the following holds. If there exists {\it any} two subcatalogs $I_1$ and $I_2$ with the property that for any bounded solution to $\mathcal{L} u = 0$ in $\om$ with $\lVert u \rVert_\infty \le 1$
\begin{equation*}
\int_{Q_0} \nn^{Q_0}u(x,\epsilon_0,I_1) \lesssim \sigma(Q_0), \quad \forall Q_0 \in \dd(\pom),
\end{equation*}
and for any bounded solution to $\mathcal{L}^* v = 0$ in $\om$ with $\lVert v \rVert_\infty \le 1$
\begin{equation*}
\int_{Q_0} \nn^{Q_0}v(x,\epsilon_0,I_2) \lesssim \sigma(Q_0), \quad \forall Q_0 \in \dd(\pom),
\end{equation*}
then $\pom$ is uniformly rectifiable.
\end{theorem}

Of course, if $L$ is self-adjoint, then only one subcatalog is required, and the condition on $v$ is redundant.

\begin{proof}[Proof of Theorem \ref{QFimpCorona.thrm}]
By definition of $U_Q$ there exists $C_{\eta,K} > 1$ such that for all $y \in Q$ and $X \in U_Q$
\begin{equation}\label{QFUR.eq1}
C_{\eta,K}^{-1} \ell(Q) \le |y - x| \le C_{\eta,K}\ell(Q).
\end{equation}
Let $a = \frac{C_{\eta,K}^{-2}}{4}$ and let $\gamma = \gamma(a)$ and $\epsilon = \epsilon'(a)$ be from Lemma \ref{constuq.lem}. Choose $M_1, M_2 \in \N$ be such that 
\begin{equation}\label{QFUR.eq2}
\begin{split}
2^{-M_1} C_{\eta,K} &< \epsilon \le 2^{-M_1 + 1}C_{\eta,K}
\\ 2^{-M_2} C_{\eta,K} &< \gamma \le 2^{-M_2 + 1}C_{\eta,K}.
\end{split}
\end{equation}
For any cube $Q \in \dd(\pom)$ let $Q(big) \in \dd(Q)$ be such that $x_Q \in Q(big)$ and $\ell(Q(big)) = 2^{-M_1} \ell(Q)$, and let $p_Q$ be an arbitrary point in $U^{i_{Q(big)}}_{Q(big)}$ where $i_{Q(big)} \in I$. By \eqref{QFUR.eq1} and \eqref{QFUR.eq2}
\begin{equation}\label{pQright.eq}
a\epsilon\ell(Q)\le  |p_Q - x_Q| \le \epsilon \ell(Q).
\end{equation}
Similarly, let $Q(little) \in \dd(Q)$ be such that $x_Q \in Q(little)$ and $\ell(Q(little)) = 2^{-M_2} \ell(Q(big))= 2^{-M_1 - M_2}\ell(Q)$, and let $s_Q$ be an arbitrary point in $U^{i_{Q(little)}}_{Q(little)}$ where $i_{Q(little)} \in I$. By \eqref{QFUR.eq1} and \eqref{QFUR.eq2}
\begin{equation}\label{sQright.eq}
|s_Q - x_Q| \le \gamma \epsilon \ell(Q).
\end{equation}

We will try to adopt the notation of \cite{AGMT} when possible and we will also drop the subscript $\mathcal{L}$ in $\hm_{\mathcal{L}}$. Let $0 < \delta \le \epsilon$ and $A \gg 1$ be fixed constants. For a fixed cube $R \in \dd(\pom)$ we say $Q \in \dd(R)$ is a high density cube and write $Q \in HD(R)$ if $Q$ is a maximal cube (with respect to containment) satisfying 
\begin{equation*}\label{hd.eq}
\frac{\hm^{p_{R}}(2Q)}{\sigma(2Q)} \ge A \frac{\hm^{p_{R}}(2R)}{\sigma(2R)}.
\end{equation*}
We say $Q \in \dd(R)$ is a low density cube and write $Q \in LD(R)$ if $Q$ is a maximal cube satisfying
\begin{equation*}\label{ld.eq}
\frac{\hm^{p_{R}}(Q)}{\sigma(Q)} \le \delta \frac{\hm^{p_{R}}(R)}{\sigma(R)}.
\end{equation*}
Next, for any cube $R \in \dd(\pom)$ we set $LD^0(R) := \{R \}$ and define $LD^k(R)$, $k \ge 1$, inductively by 
$$LD^k(R) := \bigcup_{Q \in LD^{k-1}(R)}LD(Q).$$

As in \cite{AGMT}, we may reduce the proof that $\hm$ admits a Corona decomposition to the following claim, which is analogous to \cite[Lemma 3.5]{AGMT}.

\begin{claim}\label{LDpack.cl}
If $\epsilon_0$ is sufficiently small depending on $n, K, ADR$ and $\cL$ and $0 < \delta \le \epsilon$ then for any $m \ge 1$ we have 
\begin{equation*}\label{LDpackcl.eq1}
\sum_{k = 1}^m \sum_{Q \in LD^k(R)} \sigma(Q) \le C \sigma(R),
\end{equation*}
where $C$ is independent of $m$ and $R$.
\end{claim}
\begin{proof}[Proof of claim \ref{LDpack.cl}]
Set $\F_{1,m} := \bigcup_{k =1}^m LD^{k}(R)$. We refine the collection $\F_{1,m}$, by putting some separation between cubes. We let $\F_{2,m} \subseteq \F_{1,m}$ be a collection such that if $Q, Q^* \in \F_{2,m}$ with $Q \subset Q^*$ then $\ell(Q) \le 2^{-M_2 - 1} \ell(Q^*)$ and
\begin{equation*}\label{LDpackcl.eq2}
\sum_{Q \in \F_{1,m}} \sigma(Q) \underset{M_2}{\lesssim} \sum_{Q \in \F_{2,m}} \sigma(Q).
\end{equation*}
Forming such a collection is easy. Choose the largest cube in $\F_{1,m}$, $Q$, and add it to $\F_{2,m}$ then remove from $\F_{1,m}$ all of the cubes $Q' \in \F_{1,m}$ with $Q' \subseteq Q$, $\ell(Q') \ge 2^{-M_2}\ell(Q)$. Continuing this way we obtain the collection $\F_{2,m}$. Thus, to prove the claim it is enough to show
\begin{equation}
\sum_{Q \in \F_{2,m}} \sigma(Q) \le C \sigma(R).
\end{equation}
We now produce $E_Q$ so that we may utilize Lemma \ref{constuq.lem}; we do this for all the cubes in $\F_{1,m}$ even though we will only deal with cubes in $\F_{2,m}$ later. For $Q \in \F_{1,m}$ we set $L_Q: = \cup_{Q' \in LD(Q)} Q'$ and 
$$E_Q := Q \setminus L_Q.$$ 
Since $\F_{2,m} \subseteq \F_{1,m} = \bigcup_{k =1}^m LD^{k}(R)$, we have $\{E_Q \}_{Q \in \F_{2,m}}$ are pairwise disjoint. Moreover, by definition
$$\hm^{p_Q}(L_Q) \le \sum_{Q' \in LD(Q)} \hm^{p_Q}(Q') \le \delta \sum_{Q' \in LD(Q)} \frac{\sigma(Q')}{\sigma(Q)} \hm^{p_Q}(Q) \le \delta \hm^{p_Q}(Q)$$
and hence
\begin{equation}\label{LDpackcl.eq3}
\hm^{p_Q}(E_Q) \ge (1-\delta)\hm^{p_Q}(Q) \ge (1 -\epsilon)\hm^{p_Q}(Q).
\end{equation}
By \eqref{pQright.eq}, \eqref{sQright.eq} and \eqref{LDpackcl.eq3}, we may use Lemma \ref{constuq.lem} to construct solutions $\{u_Q\}_{Q \in \F_{2,m}}$ such that 
$$u_Q(X) = \int_{\pom} f_Q \, d\hm^X, \quad 0 \le f \le \mathbbm{1}_{E_Q}$$
and
$$|u_Q(p_Q) - u_Q(s_Q)| \ge c_1 \epsilon^\alpha =: c_2.$$
Let $\Xi$ denote the collection of sequences $\{b = (b_Q): Q \in Q \in \F_{2,m}, b_Q = \pm 1\}$ and let $\lambda$ be a probability measure on $\Xi$ which assigns equal probability to $1$ and $-1$. For $b \in \Xi$ we set
$$u_b(X) = \sum_{Q \in \F_{2,m}} b_Q u_Q(X).$$
By the disjointness of $E_Q$ and the fact that $0 \le f \le \mathbbm{1}_{E_Q}$ we have
$$|u_b(X)| \le \int \sum_{Q \in \F_{2,m}} |b_Q|f_Q \, d\hm^X \le \sum_{Q \in \F{2,m}} \hm^X(E_Q) \le 1.$$
Now, using Khintchine's inequality and the construction of $u_Q$ we obtain
\begin{align*}
c_2 \le |u_Q(p_Q) - u_Q(s_Q)| &\le \left( \sum_{Q' \in \F_{2,m}} |u_{Q'}(p_Q) - u_{Q'}s(_Q)|^2 \right)^{1/2}
\\& \le \frac{1}{c_3} \int_{\Xi} \left| \sum_{Q' \in \F_{2,m}} b_{Q'}(u_{Q'}(p_Q) - u_{Q'}(s_Q))\right| \, d\lambda(b)
\\&= \frac{1}{c_3}\int_{\Xi} |u_b(p_Q) - u_b(s_Q)|\, d\lambda(b),
\end{align*}
where $c_3$ is a universal constant provided by Khintchine's inequality.
Now set $c_4: = c_2c_3$, then we have shown
\begin{equation}\label{LDpackcl.eq4}
c_4 \le \int_{\Xi}  |u_b(p_Q) - u_b(s_Q)| \, d\lambda(b).
\end{equation}
Now we prescribe  $\epsilon_0 := \frac{c_4}{16}$. Immediately, we have for any cube $Q \in \F_{2,m}$ 
\begin{equation}\label{LDpackcl.eq5}
\lambda(\{b: |u_b(p_Q) - u_b(s_Q)| > c_4/4\}) \ge c_4/8.
\end{equation}
For if not, we would contradict \eqref{LDpackcl.eq4} as the negation of \eqref{LDpackcl.eq5} leads to the estimate
$$\int_\Xi |u_b(p_Q) - u_b(s_Q)| \, d\lambda(b) \le \frac{c_4}{8}(2) + \frac{c_4}{4} \le \frac{c_4}{2},$$
where we used $|u_b(X)| \le 1$ and that $\lambda$ is a probability measure. 
For each $b \in \Xi$ and $Q \in \F_{2,m}$ we set
$$F(Q,b):=
\begin{cases}
\emptyset & \text{ if \ \ }  |u_b(s_Q) - u_b(p_Q)| \le (c_4/4) 
\\ Q(little) & \text{ if \ \ }   |u_b(s_Q) - u_b(p_Q)| > (c_4/4).
\end{cases}
$$
Then \eqref{LDpackcl.eq5} implies
$$\int_{\Xi} \int_R \mathbbm{1}_{F(Q,b)}\, d\sigma(x) , d \lambda(b) \ge \frac{c_4}{8} \sigma(Q(little)) \gtrsim \sigma(Q),$$
which implies
$$\int_{\Xi} \int_R \sum_{Q \in \F_{2,m}} \mathbbm{1}_{F(Q,b)}\, d\sigma(x) , d \lambda(b)  \gtrsim \sum_{Q \in \F_{2,m}}\sigma(Q).$$
Thus, to prove the claim it is enough to show that for all $x \in R$ and $b \in \Xi$
\begin{equation}\label{LDpackcl.eq6}
\sum_{Q \in \F_{2,m}} \mathbbm{1}_{F(Q,b)}(x) \le \nn^R u_b(x, \epsilon_0, I).
\end{equation}
Indeed, assuming \eqref{LDpackcl.eq6} and \eqref{QFholdsforQFimpUR.eq}  we have
\begin{align*}
\sum_{\F_{2,m}} \sigma(Q) &\lesssim \int_{\Xi} \int_R \sum_{Q \in \F_{2,m}} \mathbbm{1}_{F(Q,b)}\, d\sigma(x) , d \lambda(b)
\\&  \lesssim \int_{\Xi} \int_R \nn^R u_b(x, \epsilon_0, I)\, d\sigma(x) , d \lambda(b)
\\& \lesssim \sigma(R),
\end{align*}
where we recall that $\lVert u_b \rVert_\infty \le 1$.

We turn our attention to showing \eqref{LDpackcl.eq6}, recalling that $\epsilon_0 = c_4/16$. For notational convenience, the cubes we will use to `test' in the definition of $\nn$ will be with a superscript instead of a subscript. 
Fix $b \in \Xi$ and $x \in R$. Let $\{Q_j\}_{j = 1}^{j_0}$, where $j_0 \le m$ be the cubes in $\F_{2,m}$ such that $\mathbbm{1}_{F(Q_j,b)}(x) = 1$, that is, the collection of $Q_j \in \F_{2,m}$ such that $x \in Q(little)$. 
Note
$$j_0 = \sum_{Q \in \F_{2,m}} \mathbbm{1}_{F(Q,b)}(x).$$ 
We relabel $\{Q_j\}$ so that $Q_{j + 1} \subset Q_j$. For each $j$, let $Q_j^* = Q_j(big)$ and $Q_j' = Q_j(little)$. Then by construction of $\F_{2,m}$ for $j = 1, \dots, j_0 - 1$ we have 
$$x \in Q_{j+1}' \subset Q_{j +1}^* \subset Q_j' \subset Q_j^*.$$
Indeed, all but the `middle' inclusion, $Q_{j +1}^* \subset Q_j'$, is obvious; however by construction of $\F_{2,m}$
$$\ell(Q_{j+1}^*) = 2^{-M_1}\ell(Q_{j+1}) \le 2^{-M_1 - M_2 - 1} \ell(Q_j) = 2^{-1} \ell(Q_j').$$

Now we choose $Q^k$ and $X_k$ to 
obtain a lower bound on $\nn^R u_b(x, \epsilon_0, I)$. Set 
$X_0 = p_{Q_1}$ and $X_1 = s_{Q_1}$, $Q^0 = Q_1^*$ and 
$Q^1 = Q_1'$ so that $|u_b(X_0) - u_b(X_1)| > c_4/4 > c_4/16$, $X_0 \in U_{Q^0}^{i_{Q^0}}$ and 
$X_1 \in U_{Q^1}^{i_{Q^1}}$. For $k \ge 2$, having chosen $X_{k -1}$ we choose $X_k$ in the 
following way. We have that $|u_b(p_{Q_k}) - u_b(s_{Q_k})| > c_4/4$ so that either
$$|u_b(p_{Q_k}) - u_b(X_{k-1})| > c_4/16$$
or
$$|u_b(s_{Q_k}) - u_b(X_{k-1})| > c_4/16.$$
Thus, we may choose $X_k \in \{p_{Q_k}, s_{Q_k}\}$ so that
$$|u_b(X_k) - u_b(X_{k-1})| > c_4/16.$$
We then choose $Q^k$ to be $Q_k^*$ if $X_k = p_{Q_k}$ and $Q^k$ to be $Q_k'$ if $X_k = s_{Q_k}$. 
Having done this for $k = 2,\dots j_0$ we obtain 
$\{Q^k\}_{k =0}^{j_0}$ where $x \in Q^{k+1} \subset Q^k \subseteq R$ and 
$X_k \in U_{Q^k}^{i_{Q^k}}$ are such that 
$$|u_b(X_k) - u_b(X_{k -1})| > c_4/16.$$
It follows that 
$$\sum_{Q \in \F_{2,m}} \mathbbm{1}_{F(Q,b)}(x) = j_0 \le \nn^R u_b(x, \epsilon_0, I),$$
which is \eqref{LDpackcl.eq6}. This proves the claim.
\end{proof}
With claim \ref{LDpackcl.eq1} in hand, the proof then proceeds exactly as in \cite{AGMT}. 
\end{proof}

\section{Uniform Rectifiability Implies A Quantitative Fatou Theorem}\label{s4}

The converse to Theorem \ref{QFimpUR.thrm} (which therefore completes Theorem \ref{tmain}) 
is the following.

\begin{theorem}\label{dydfatthrm.thrm}
Let $\om \subset \ree$ be an open set satisfying the corkscrew condition with 
$n$-dimensional UR boundary $\pom$, 
and let $\cL$ be a (real) divergence form elliptic operator with coefficients satisfying 
\eqref{admisop.eq}\footnote{Or even \eqref{eq2.18}; see Remark \ref{remark2.16}.} and 
\eqref{admisop2.eq}. For all $\eta \ll 1 \ll K$ (with $\eta \ll K^{-1}$) and $\tau \in (0, \tau_0/2]$ the 
following holds:  there exists a subcatalog $I$ such that for every bounded solution to $\cL u =$ in $\om$ 
with $\lVert u \rVert_{L^\infty(\om)} \le 1$,
\begin{equation}\label{dydfatthrmconc.eq}
\int_{Q_0} \nn^{Q_0}u(x,\epsilon, I)\, d\sigma(x) \lesssim \sigma(Q_0) \quad \forall Q_0 \in \dd(\pom),
\end{equation}
where implicit constants depend on $\epsilon, \eta, K, \tau$, UR/ADR, the DeG/N/M constants and the constants in \eqref{admisop.eq} and \eqref{admisop2.eq} (but not $Q_0, u$ or $I$). 
\end{theorem}

The following pair of lemmata lie at the heart of the proof of Theorem \ref{dydfatthrm.thrm}.

\begin{lemma}[{\cite[Lemma 2.2]{HMM}}]\label{HMMlemma1.lem} Suppose that $E\subset \ree$ is $n$-dimensional UR.  Then given any positive constants
$\eta\ll 1$
and $K\gg 1$, 
there is a disjoint decomposition
$\dd(E) = \G\cup\B$, satisfying the following properties.
\begin{enumerate}
\item  The ``Good"collection $\G$ is further subdivided into
disjoint stopping time regimes, such that each regime $\{\sbf\}$ is coherent (see Definition \ref{d3.11}).

\item The ``Bad" cubes, as well as the maximal cubes $Q(\sbf)$ satisfy a Carleson
packing condition:
\begin{equation}\label{bilateralcubespack.eq}
\sum_{Q'\subseteq Q, \,Q'\in\B} \sigma(Q')
\,\,+\,\sum_{\sbf: Q(\sbf)\subseteq Q}\sigma\big(Q(\sbf)\big)\,\leq\, C_{\eta,K}\, \sigma(Q)\,,
\quad \forall Q\in \dd(E)\,.
\end{equation}
\item For each $\sbf$, there is a Lipschitz graph $\Gamma_{\sbf}$, with Lipschitz constant
at most $\eta$, such that, for every $Q\in \sbf$,
\begin{equation*}\label{eq2.2a}
\sup_{x\in \Delta_Q^*} \dist(x,\Gamma_{\sbf} )\,
+\,\sup_{y\in B_Q^*\cap\Gamma_{\sbf}}\dist(y,E) < \eta\,\ell(Q)\,,
\end{equation*}
where $B_Q^*= B(x_Q,K\ell(Q))$ and $\Delta_Q^*:= B_Q^*\cap E$.
\end{enumerate}
\end{lemma}

Next, we recall a construction in  \cite[Section 3]{HMM}, leading up to and including in particular
\cite[Lemma 3.24]{HMM}, which says that for a $UR$ set $E$, the open set $\Omega_E:= \ree\setminus E$ 
has an approximation, of Corona type, by Chord-arc domains (Definition \ref{def1.cad}).   
We summarize this construction as follows. 
\begin{lemma}\label{lemma2.7}
Let $E\subset \ree$ be 
an $n$-dimensional UR  set, and let $\om_E:= \ree\setminus E$.  Given positive constants
$\eta\ll 1$
and $K\gg 1$, as in \eqref{eq3.1} and Remark \ref{remark:E-cks},  
let $\dd(E) = \G\cup\B$, be the corresponding 
bilateral Corona decomposition of Lemma \ref{HMMlemma1.lem}. 
Then for each $\sbf\subset \G$, and for each $Q\in \sbf$, the collection 
$\W^0_Q$ in \eqref{eq3.1} (defined with respect to the open set $\Omega_E$), 
has an augmentation $\W^*_Q\subset \W$ satisfying the following properties.
\begin{enumerate}
\item $\W^0_Q\subset \W^*_Q = \W_Q^{*,+}\cup\W_Q^{*,-}$,
where (after a suitable rotation of coordinates)
each $J \in \W_Q^{*,+}$ lies above the Lipschitz graph $\Gamma_{\sbf}$
of Lemma \ref{HMMlemma1.lem},  each $J \in \W_Q^{*,-}$ lies below $\Gamma_{\sbf}$.
Moreover, if $Q'$ is a child of $Q$, also belonging to
$\sbf$, then every $J\in \W_Q^{*,+}\cup \W_{Q'}^{*,+}$ (resp. $\W_Q^{*,-}\cup \W_{Q'}^{*,-}$) 
is contained in the same connected
component of $\om_E$,
and  $\W_{Q'}^{*,+}\cap \W_{Q}^{*,+}\neq \emptyset$ (resp.,
$\W_{Q'}^{*,-}\cap\W_{Q}^{*,-}\neq \emptyset$). 

\smallskip
\item There are uniform constants $c$ and $C$ such that $\forall J\in \mathcal{W}^*_Q$,
\begin{equation}\label{eq2.whitney2}
\begin{array}{c}
c\eta^{1/2} \ell(Q)\leq \ell(J) \leq CK^{1/2}\ell(Q)\,, 
\\[5pt]
\dist(J,Q)\leq CK^{1/2} \ell(Q)\,,
\\[5pt]
c\eta^{1/2} \ell(Q)\leq\dist(J^*(\tau),\Gamma_{\sbf})\approx \dist\big(J^*(\tau),E\big)\,,\quad \forall 
\tau\in (0,\tau_0]\,.
\end{array}
\end{equation}
\end{enumerate}

Moreover, given $\tau\in(0,\tau_0]$ (with $\tau_0$ as in previous sections), 
set
\begin{equation}\label{eq3.3aa}
\tuq^\pm=\widetilde{U}^\pm_{Q,\tau}:= \bigcup_{J\in \W^{*,\pm}_Q} {\rm int}\left(J^*(\tau)\right)\,,
\qquad \tuq:= \tuq^+\cup \tuq^-\,,
\end{equation}
and given $\sbf'$, a semi-coherent subregime of $\sbf$, define \begin{equation}\label{eq3.2}
\Omega_{\sbf'}^\pm = \Omega_{\sbf'}^\pm(\tau) := \bigcup_{Q\in\sbf'} \tuq^{\pm}\,.
\end{equation}
Then 
each of $\Omega^\pm_{\sbf'}$ is a CAD, with Chord-arc constants
depending only on $n,\tau,\eta, K$, and the
ADR/UR constants for $\pom$.  

Finally, for $Q\in \G$, if $U_Q$ is defined as in \eqref{eq.UQdef}, then each connected component of
$U_Q$ is contained in either $\tuq^+$ or in $\tuq^-$, and conversely, each component
of $\tuq$ contains at least one component of $U_Q$.
\end{lemma}

We mention that the Whitney regions $\tuq$ were simply denoted $U_Q$ in \cite{HMM}, but for 
our purposes in the present section, we prefer to avoid
conflict with the notation introduced in \eqref{eq.UQdef}.


\begin{remark}\label{remark2.12} In particular, for each $\sbf\subset \G$,
if $Q'$ and $Q$ belong to $\sbf$, and 
if $Q'$ is a dyadic child of $Q$, then $\widetilde{U}_{Q'}^+\cup \tuq^+$ is Harnack Chain connected,
and every pair of points $X,Y\in \widetilde{U}_{Q'}^+\cup \tuq^+$ may be connected by a Harnack Chain 
 in $\Omega_E$ 
of length at most $C= C(n,\tau,\eta,K,\textup{ADR/UR})$.  The same is true for  
$\widetilde{U}_{Q'}^-\cup \tuq^-$.
\end{remark}

\begin{remark}\label{remark4.10} Note that by \eqref{eq2.whitney2}, we have in particular that
$\delta(Y) \approx \ell(Q)$, for all $Y\in \tuqt$, provided that $\tau \leq \tau_0/2$.
\end{remark}

\begin{remark}\label{remark2.13} Let $0<\tau\leq \tau_0/2$.
Given any $\sbf\subset \G$, and any semi-coherent subregime
$\sbf'\subset \sbf$, define $\om_{\sbf'}^\pm=\om_{\sbf'}^\pm(\tau)$ as in \eqref{eq3.2},
and similarly set $\widehat{\om}_{\sbf'}^\pm=\om_{\sbf'}^\pm(2\tau)$.  Then by construction, for
any $X\in \overline{\om_{\sbf'}^\pm}$, 
$$ 
\dist(X,E) \approx \dist(X, \partial \widehat{\om}_{\sbf'}^\pm)\,,$$
where of course the implicit constants depend on $\tau$.
\end{remark}

As in \cite{HMM}, it will be useful for us to extend the definition of the Whitney region $\tuq$ to the case that
$Q\in\B$, the ``bad" collection of Lemma \ref{HMMlemma1.lem} and Lemma \ref{lemma2.7}.   Let $\W_Q^*$ be the augmentation of $\W_Q^0$
as constructed in Lemma \ref{lemma2.7}, and set
\begin{equation}\label{Wdef}
\W_Q:=\left\{
\begin{array}{l}
\W_Q^*\,,
\,\,Q\in\G,
\\[6pt]
\W_Q^0\,,
\,\,Q\in\B
\end{array}
\right.\,.
\end{equation}
For $Q \in\G$ we shall therefore sometimes simply write $\W_Q^\pm$ in place of $\W_Q^{*,\pm}$.
For arbitrary $Q\in \dd(E)$, we may then define
\begin{equation}\label{eq3.3bb}
\tuq=\widetilde{U}_{Q,\tau}:= \bigcup_{I\in \W_Q} {\rm int}\left(J^*(\tau)\right)\,.
\end{equation}
Let us note that for $Q\in\G$, the latter definition agrees with that in \eqref{eq3.3aa}.
On the other hand, for $Q \in \B$, the regions defined in \eqref{eq3.3bb} are precisely the same as
the regions $U_Q$ defined in \eqref{eq.UQdef}.  
We then have the following.

\begin{remark}\label{r2.38} We note that for any fixed $\eta$ and $K$, and for any given $Q$,
$U_Q\subset \tuq$, and
 there exists $N = N(\eta, K)$ such that 
 $\tuq$ has at most $N$ connected components.
\end{remark}

\begin{remark}\label{remark2.39}
Given an open set
$\Omega$ satisfying an interior Corkscrew condition, for each
cube $Q\in\dd(\pom)$, we construct the Whitney regions $\tuq$ relative to the set $E:=\pom$,
and we note that at least one component of each $\tuq$ is contained in $\Omega$,
provided that we choose
 $\eta$ small enough and $K$ large enough, depending on the constant in the Corkscrew condition.
 We shall henceforth {\bf always} choose $\eta$ and $K$ accordingly. 
 \end{remark}
 
 We now define subcatalogs and a counting function adapted to the regions $\tuq$.
 
 \begin{definition}[Special Subcatalogs]\label{def4.15}
 Let $\om \subset \ree$ be an open set 
 satisfying the (interior) corkscrew condition such that $\pom$ is 
UR. Given $\eta \ll 1 \ll K$, we define the augmented Whitney regions
$\tuq$ as in Lemma \ref{lemma2.7} and \eqref{Wdef} - \eqref{eq3.3bb}, and 
we enumerate the connected
 components of each $\tuq$ as $\{\widetilde{U}_Q^i\}_{i}$.
 We say that $\widetilde{I}$ is a special subcatalog, 
 if $\widetilde{I} = \{i_Q\}_{Q \in \dd(\pom)}$, where for each $Q$, $\tuq^{i_Q}$
 is one of the enumerated components of $\tuq$; 
 i.e., in a special subcatalog,
we have fixed precisely one component of $\tuq$
for each $Q$.

\end{definition}

\begin{definition}[Admissible sequences and the Special Dyadic Oscillation Counting Function] \label{def4.16}
Let $\om \subset \ree$ be an open set satisfying the (interior) corkscrew condition such that $\pom$ is 
UR, and let $\eta \ll 1 \ll K$.  Let $u : \om \to \re$. Given a special subcatalog $\widetilde{I}$, a cube $Q \in \dd(\pom)$, a 
point $x\in Q$, and a 
number $\epsilon > 0$, we say that a sequence $\{X_k\}_{k=1}^{k_0+1}\subset \Omega$,
of arbitrary finite cardinality $k_0+1\geq 2$, 
is $(x,\epsilon,\widetilde{I},Q)$-{\it admissible for} $u$ (or simply  $(x,\epsilon,\widetilde{I},Q)$-admissible
when $u$ is understood from context)
if there exist strictly nested
cubes $\{Q_k\}_{k=1}^{k_0+1}$ with $x\in Q_{k_0+1}\subsetneq Q_{k_0}\subsetneq ...
\subsetneq Q_1  \subseteq Q$,
such that $X_k \in \widetilde{U}_{Q_k}^{i_{Q_k}}$ with $i_{Q_k} \in \widetilde{I}$, 
and $|u(X_k) -u(X_{k+1})|>\epsilon$. 

The special dyadic oscillation counting function is then defined to be
\begin{equation}
\nnt^{Q}u(x,\epsilon, \widetilde{I}\,)
:= \sup\{k_0 : \exists \, (x,\epsilon,\widetilde{I},Q)\text{-admissible} \{X_k\}_{k=1}^{k_0+1}
\}\,.
\end{equation}
If there is no such $(x,\epsilon,\widetilde{I},Q)$-admissible sequence of cardinality at least 2, we 
set  $\nnt^{Q}u(x,\epsilon, \widetilde{I}\,)=0$.
\end{definition}

Given a subcatalog $I=\{i_Q\}_Q$, and a special subcatalog $\widetilde{I}=\{j_Q\}_Q$, we say that
$$I\prec \widetilde{I} \quad \text{if}\quad U_Q^{i_Q}\subset \tuq^{j_Q}\,,
\quad \forall \, Q\,.$$
We note that if $I\prec \widetilde{I}$, then 
$$\nn^Q(x,\epsilon,I) \leq \nnt^Q(x,\epsilon,\widetilde{I}\,)\,,$$
for every $Q\in\dd(\pom)$, $x\in Q$, and  $\epsilon >0$;  indeed, if $I\prec \widetilde{I}$, then
every $(x,\epsilon,I,Q)$-admissible sequence $\{X_k\}_{k=1}^{k_0+1}$ is also
$(x,\epsilon,\widetilde{I},Q)$-admissible.

Moreover, by Lemma \ref{lemma2.7}, for every $Q$, each component of
$\tuq$ contains at least one component of $U_Q$.  Thus, for every
special subcatalog $\widetilde{I}$, there is at least one subcatalog
$I$ with $I\prec \widetilde{I}$.  Consequently, Theorem \ref{dydfatthrm.thrm}
is an immediate corollary of the following slightly stronger version of itself, which is 
really the main result of this section.

\begin{theorem}\label{t4.18}
Let $\om \subset \ree$ be an open set satisfying the corkscrew condition, 
with $n$-dimensional UR boundary $\pom$, 
and let $\cL$ be a (real) divergence form elliptic operator with coefficients satisfying 
\eqref{admisop.eq} (or  \eqref{eq2.18}) and 
\eqref{admisop2.eq}. For all $\eta \ll 1 \ll K$ (with $\eta \ll K^{-1}$) and $\tau \in (0, \tau_0/2]$ the 
following holds:  there exists a special
subcatalog $\widetilde{I}$ such that for every bounded solution 
to $\cL u =$ in $\om$ with $\lVert u \rVert_{L^\infty(\om)} \le 1$,
\begin{equation}\label{eq4.19}
\int_{Q_0} \nnt^{Q_0}u (x,\epsilon, \widetilde{I}\,)\, 
d\sigma(x) \lesssim \sigma(Q_0)\,, \quad \forall Q_0 \in \dd(\pom),
\end{equation}
where the implicit constants depend on $\epsilon, \eta, K, \tau$, UR/ADR, the DeG/N/M constants and the constants in \eqref{admisop.eq} (or \eqref{eq2.18}) and \eqref{admisop2.eq} (but not $Q_0, u$ or $\widetilde{I}$). 
\end{theorem}

\begin{proof}[Proof of Theorem \ref{t4.18}] 
For suitably small positive $\eta\ll 1$, and suitably large $K\gg\eta^{-1}$, taking $E:=\pom$,
we make the corona decomposition $\dd(E) =\G \cup\B$
of Lemma \ref{HMMlemma1.lem} and Lemma \ref{lemma2.7}, and we
construct the chord-arc domains $\Omega^\pm_{\sbf}$ associated to each stopping time regime
$\sbf\subset \G$.     Since $\Omega$ satisfies an interior Corkscrew condition, for each
cube $Q\in\dd(E)$, at least one component of the Whitney region $\tuq$ is contained in $\Omega$,
provided that we choose
 $\eta$ small enough and $K$ large enough, depending on the constant in the Corkscrew condition. 
Consequently, for each stopping time regime $\sbf\subset \G$, at least one of
 $\Omega^+_{\sbf}$ or $\Omega^-_{\sbf}$ must be contained in $\Omega$.

 We begin by building an appropriate subcatalog. For each $Q\in\dd(E)$, 
 we renumber the components of $\tuq$ so that $\tuq^1\subset \Omega$.  As noted above,
 there is always at least one such component.   For the bad cubes $Q\in \B$, if there is 
 more than one such component, then we simply fix one arbitrarily.
 For the good cubes, we do this in a more precise 
 manner:
 if $Q \in \G$, then $Q$ belongs to some particular
 stopping time regime $\sbf$, and its augmented Whitney region $\tuq$ has exactly two components
 $\tuq^+$ and $\tuq^-$.    If only one of these is contained in
 $\Omega$,  we then let $\tuq^1$ denote that component;  if both $\tuq^\pm\subset\Omega$,  
 we then arbitrarily
 set $\tuq^+ =: \tuq^1$.   Note that by construction, for a given stopping time regime $\sbf$, 
 there is a component of $\Omega_\sbf$, which we now relabel
 $\Omega_\sbf^1:= \cup_{Q\in\sbf} \tuq^1$, which is contained in $\Omega$.
 Having globally renumbered the components of every $\tuq$ in this way, we now set
 $i_Q:=1$ for each $Q$, and define the special subcatalog $\widetilde{I}:=\{i_Q\}_Q$.

 Having fixed $\widetilde{I}$ for the remainder of this section, 
 we henceforth suppress the dependence of $\nnt$ on $\widetilde{I}$.  
 Fix $Q_0 \in \dd(E)$ and $x \in Q_0$. We set some useful notation. For every $\sbf$, if $x \in Q \in \dd(Q_0) \cap \sbf$ we set $Q_{max}(\sbf, x, Q_0)$ to be $Q(\sbf)$ if $Q(\sbf) \subset Q_0$ and $Q_0$ otherwise. Note that the latter can only happen for one stopping time regime as $Q \subseteq Q_0 \subseteq Q(\sbf)$ implies $Q_0 \in \sbf$. 
 Similarly, if $x \in Q \in \dd(Q_0) \cap \sbf$ we set $Q_{min}(\sbf, x, Q_0)$ be the smallest 
 cube in $\dd(Q_0) \cap \sbf$ such that $x \in Q$; 
 if there is no smallest cube, we abuse notation and set $Q_{min}(\sbf, x, Q_0) ``=" x$. 

Next, we define doubly truncated $\nnt$. If $Q', Q^* \in \sbf$ for some $\sbf$ with $Q' \subset Q^*$, and 
$x \in Q'$ we  say that a sequence $\{X_k\}_{k=1}^{k_0+1}\subset \Omega$,
of arbitrary finite cardinality $k_0+1\geq 2$, 
is $(x,\epsilon,Q',Q^*)$-{\it admissible} if there exist strictly nested
cubes $\{Q_k\}_{k=1}^{k_0+1}$ with $x\in Q'\subseteq Q_{k_0+1}\subsetneq Q_{k_0}\subsetneq ...
\subsetneq Q_1  \subseteq Q^*$,
such that $X_k \in \widetilde{U}_{Q_k}^{i_{Q_k}}$ with 
$i_{Q_k} \in \widetilde{I}$, and $|u(X_k) -u(X_{k+1})|>\epsilon$.

Abusing notation, we allow $Q'$ to be ``equal" to $x$,
and in this case the 
sequence is admissible in the sense of Definition \ref{def4.16}.

The doubly truncated dyadic oscillation counting function is then defined to be
\begin{equation}\label{eq4.20}
\nnt^{Q^*}_{Q'}u(x,\epsilon, \widetilde{I}\,)
:= \sup\{k_0 : \exists \, (x,\epsilon,Q',Q^*)\text{-admissible} \{X_k\}_{k=1}^{k_0+1}
\}\,.
\end{equation}
If there is no such $(x,\epsilon,Q',Q^*)$-admissible sequence of cardinality at least 2, we 
set  $\nnt^{Q^*}_{Q'}u(x,\epsilon, \widetilde{I}\,)=0$.

Using the same abuse of notation as above, we allow $Q'$ to be ``equal" to $x$,
and in this case we 
simply have
$\nnt^{Q^*}_x (x,\epsilon) = \nnt^{Q^*}(x,\epsilon)$.

\begin{claim}\label{dydfatthrmcl1.cl}
\begin{multline}\label{dydfatcl1conc.eq}
\nnt^{Q_0}u(x,\epsilon) \le \sum_{\sbf: \exists Q: \dd(Q_0) \cap \sbf \ni Q 
\ni x}\nnt^{Q_{max}(\sbf, x, Q_0)}_{Q_{min}(\sbf, x, Q_0)}u(x,\epsilon)\, + 
\sum_{\substack{Q:\, x\in Q \in \B\\ Q \subseteq Q_0}} 1 \,+ \sum_{\substack{\sbf:\,Q(\sbf) \ni x \\
Q(\sbf) \subseteq Q_0}} 1 \\[4pt]
=: \,{\sum}_1 \,+ \,{\sum}_2 \,+\,{\sum}_3\,,
\end{multline}
where of course ${\sum}_j={\sum}_j(x)$, for each $j = 1,2,3$.
\end{claim}
\begin{proof}[Proof of claim \ref{dydfatthrmcl1.cl}]
Consider any $(x,\epsilon,\widetilde{I},Q_0)$-admissible sequence $\{X_k\}_{k=1}^{k_0+1}$.
By definition (and by the construction of this particular $\widetilde{I}$), this means that 
$X_k\in \widetilde{U}_{Q_k}^1$, with $x\in  Q_{k_0+1}\subsetneq Q_{k_0}\subsetneq ...
\subsetneq Q_1  \subseteq Q_0$.  Trivially,
$$\#\{k: \, Q_k \in \B\} \leq {\sum}_2 \,,$$
so we need only treat $\{k:\, Q_k\in \G\} =\cup_\sbf \{k:\, Q_k \in \sbf\}$.  Consider now any $\sbf$ which 
contains at least one $Q_k$.  In this case, $\#\{k:\, Q_k \in \sbf\} = m+1$, for some $m\geq 0$, 
hence, for this particular $\sbf$, there is an $\left(x,\epsilon,Q_{min}(\sbf, x, Q_0),Q_{max}(\sbf, x, Q_0)
\right)$-admissible sequence of cardinality $m+1$.  Consequently, by definition,
$$m \leq \nnt^{Q_{max}(\sbf, x, Q_0)}_{Q_{min}(\sbf, x, Q_0)}u(x,\epsilon)\,.$$ 
Adding 1 to both sides of the last inequality, and summing in $\sbf$, we find that 
$$\#\{k:\, Q_k\in \G\}\, \leq\,  {\sum}_1 + {\sum}_3 + 1\,.$$
The extra 1 on the right hand side accounts for the case that there may be one $\sbf$ for which
$Q_0\subsetneq Q(\sbf)$ (so $Q_{max}(\sbf, x, Q_0) = Q_0$), which case is not included in 
${\sum}_3$.
The claim now follows
when we take a supremum over all $(x,\epsilon,\widetilde{I},Q_0)$-admissible sequences.
\end{proof}
Next we observe that
\begin{multline*}
\int_{Q_0} \left({\sum}_2 + {\sum}_3\right) d\sigma\\[4pt]
= \int_{Q_0} \left(\sum_{Q \in \B:\, Q \subseteq Q_0} 1_Q(x) \,+ 
\sum_{\sbf:\,
Q(\sbf) \subseteq Q_0} 1_{Q(\sbf)}(x)\right)d\sigma(x)
\\[4pt]=\, \sum_{Q \in \B:\, Q \subseteq Q_0} \sigma(Q) \,+ 
\sum_{\sbf:\,
Q(\sbf) \subseteq Q_0} \sigma\big(Q(\sbf)\big) \lesssim \sigma(Q_0)\,,
\end{multline*}
where in the last step we have used the packing condition on $\B$ and on $\{Q(\sbf)\}_\sbf$, i.e., 
\eqref{bilateralcubespack.eq}.  Consequently, to complete the proof 
of \eqref{eq4.19}, it remains to treat ${\sum}_1$.
To this end, we first define truncated dyadic ``cones"
$$\Gamma^{Q_0}(x) := \bigcup_{Q:\, x\in Q \subseteq Q_0} \tuqt\,,$$
and for  $Q' \subseteq Q^*$, doubly truncated dyadic cones
$$\Gamma_{Q'}^{Q^*} := \bigcup_{Q:\, Q'\subseteq Q \subseteq Q^* } \tuqt\,,$$
where the fattened Whitney regions $\tuqt$ are defined as in \eqref{eq3.3bb}, but 
with $2\tau$ in place of $\tau$, and we recall that we have fixed $\tau\leq \tau_0/2$.

\begin{claim}\label{dydfatthrmcl2.cl} Let $\varphi$ be an $\epsilon/8$ approximation of $u$ as in Definition  \ref{Cestepsapprox.eq} (afforded by Theorem \ref{epsapproxthrm.thrm}). Then
\begin{equation}\label{dydfatthrmcl2conc.eq}
{\sum}_1\,(x)\, \lesssim \,\int_{\Gamma^{Q_0}(x)} |\nabla \varphi(Y)| \delta(Y)^{-n} \, dY.
\end{equation}
\end{claim}
Momentarily taking the claim for granted, by the definition of $\Gamma^{Q_0}$
and Remark \ref{remark4.10},
we see that
\begin{multline*}
\int_{Q_0} {\sum}_1 \, d\sigma\,\lesssim
\int_{Q_0} \int_{\Gamma^{Q_0}(x)} |\nabla \varphi(Y)| \delta(Y)^{-n} \, dY \, d\sigma(x)\\[4pt]
\lesssim\, \sum_{Q\subseteq Q_0}\int_{Q_0} 1_Q(x) \int_{\tuqt}
 |\nabla \varphi(Y)|\,\ell(Q)^{-n}  \, dY \, d\sigma(x)\\[4pt] \approx \,\int_{T_{Q_0,2\tau}} 
  |\nabla \varphi(Y)|\,dY\, \lesssim\,
 \sigma(Q_0)
\end{multline*}
where $T_{Q_0,2\tau}:= \cup_{Q\subseteq Q_0} \tuqt$ is a dyadic Carleson region of diameter
$d\approx\ell(Q_0)$, and  where 
we have used  Fubini's Theorem, then the bounded overlap property of the Whitney regions
$\tuqt$, and then Theorem \ref{epsapproxthrm.thrm} and the 
definition of $\epsilon$-approximability (Definition
\ref{epsapprox.def}).  Estimate \eqref{eq4.19}, and hence the conclusion of Theorem \ref{t4.18} 
now follow.

It therefore remains only to prove Claim \ref{dydfatthrmcl2.cl}.
In turn, by the definition of ${\sum}_1$, and the bounded overlap property of
the Whitney regions $\tuqt$, Claim \ref{dydfatthrmcl2.cl} is an immediate consequence 
of the following claim.
\begin{claim}\label{claim4.23}
For each $\sbf$ 
$$\nnt^{Q_{max}(\sbf, x, Q_0)}_{Q_{min}(\sbf, x, Q_0)}u(x,\epsilon) \,\lesssim\,
\int_{\Gamma^{Q_{max}(\sbf, x, Q_0)}_{Q_{min}(\sbf, x, Q_0)}} 
|\nabla \varphi(Y)| \delta(Y)^{-n} \, dY.$$
\end{claim}

It therefore remains only to prove Claim \ref{claim4.23}.

\begin{proof}[Proof of Claim \ref{claim4.23}]
We begin with a preliminary observation, for future reference.  Given a point $X\in \Omega$, let 
$B_\gamma^X:= B\big(X,\gamma \delta(X)\big)$, where $\gamma >0$ is a small number to be 
chosen momentarily.
By the De Giorgi-Nash-Moser estimates, and the fact that $\|u\|_\infty \leq 1$,
\begin{equation*}
|u(X) - \fint_{B} u|\,\leq\, C \,\gamma^{\,\beta}\, \leq\, \frac{\epsilon}{8}\,,
\end{equation*}
for any $B\subseteq B_\gamma^X$,
by choice of $\gamma=\gamma(\epsilon)$ small enough.  Let $\vp$ be the $\epsilon/8$ approximation
of $u$ (Definition \ref{epsapprox.def}), 
whose existence is guaranteed by Theorem \ref{epsapproxthrm.thrm}.  Then
\begin{equation}\label{eq4.27}
|u(X) - \fint_{B} \vp|\, \leq\, \frac{\epsilon}{4}\,,\qquad \forall \, B\subseteq B_\gamma^X\,.
\end{equation}

We now fix some $\sbf$ which meets $\dd(Q_0)$, and let $x\in Q_0$.    To simplify notation,
set $Q_{max}:=Q_{max}(\sbf, x, Q_0)$, $Q_{min}:= Q_{min}(\sbf, x, Q_0)$.  Consider now
any $\left(x,\epsilon,Q_{min},Q_{max}\right)$-admissible 
sequence $\{X_k\}_{k=1}^{k_0+1}$, of cardinality $k_0+1\geq 2$; if there is no 
such sequence then there is nothing to prove for this $\sbf$.
By definition, $X_k\in \widetilde{U}_{Q_k}^1$, and for consecutive points
$X_{k+1}$ and $X_k$, we have 
$$Q_{min}\subset Q_{k+1}\subsetneq Q_k\subset Q_{max}\,.$$ 
We now form a chain of cubes $\{P_j\}_{j=1}^{N(k)}$ such that
$$Q_k = :P_1\supset P_2 \supset ...\supset P_{N(k)} := Q_{k+1}\,,$$
and such that $P_j$ is the dyadic parent of $P_{j+1}$.  Let $Y_j$ denote the 
center of some $J \in \W^1_{P_j}$,
where for each cube $Q$, $\W_Q^1:=\{J\in \W_Q:\,J^* \subset \tuq^1\}$.
By the chord arc property of $\Omega^1_{\sbf}$, and Remark \ref{remark2.13},
we may connect $X_k$ to $Y_1$ (resp., $X_{k+1}$ to $Y_{N(k)}$), by 
a Harnack chain of balls $\{B_i\}_{i=1}^{M}$, of uniformly bounded cardinality $M$
no larger than some $M_0$
depending only on $\tau$ and dimension,  in such a way that
the radius of each ball is
comparable (again depending on $\tau$)
to $\ell(Q_k)$ (resp., $\ell(Q_{k+1})$), and such that any two consecutive balls in the Harnack chain,
say $B_i$ and $B_{i+1}$,
are contained in a ball $\bt_i$ whose double 
is contained in $\widetilde{U}_{Q_k,2\tau}$
(resp., $\widetilde{U}_{Q_{k+1},2\tau}$).  
In particular, note that by construction, each $\bt_i$ is contained in $\Gamma_{Q_{min}}^{Q_{max}}$,
and that $\delta(Y) \approx r_{\bt_i}$, the radius of $\bt_i$, for every $Y\in \bt_i$.

Moreover, we note that the 
terminal ball in each of these chains, 
containing $X_k$ (resp. $X_{k+1}$), 
may be taken to be 
centered at $X_k$ (resp. $X_{k+1}$), 
and may be chosen to have small enough radius, depending on $\gamma$, but still comparable to
$\delta(X_k) \approx \ell(Q_k)$ (resp. $\delta(X_{k+1}) \approx \ell(Q_{k+1})$), such that
\eqref{eq4.27} holds with $X=X_k$ and with $X= X_{k+1}$.

In addition, for each $j = 1,2,...,N(k)-1$,
by Remark \ref{remark2.12}, we may connect $Y_j$ to $Y_{j+1}$ by another
Harnack chain of uniformly bounded cardinality, such that the radius of each ball is comparable
to $\ell(P_j)$, and such that any two consecutive balls in the Harnack chain,
again call them $B_i$ and $B_{i+1}$,
are contained in a ball $\bt_i$ whose double 
is contained in $\widetilde{U}_{P_j,2\tau}\cup \widetilde{U}_{P_{j+1},2\tau}$.
As above, each $\bt_i$ is contained in $\Gamma_{Q_{min}}^{Q_{max}}$, and
 $\delta(Y) \approx r_{\bt_i}$, for every $Y\in \bt_i$.

Combining these chains, we obtain a 
Harnack chain
$\{B^k_{i,j}\}_{1\leq i\leq M(j),\,0\leq j\leq N(k)}$
with reverse lexicographical ordering
$B^k_{1,1}, B^k_{2,1},...,B^k_{M(1),1}, B^k_{1,2},...B^k_{M(2),2},...$ etc., 
joining $X_k$ to $X_{k+1}$,  such that $M(j)\leq M_0$ for each $j$,
and such that 
consecutive balls  $B^k_{i,j}$ and $B^k_{i+1,j}$
are contained in a ball $\bt^k_{i,j}\subset \Gamma_{Q_{min}}^{Q_{max}}$. 
Moreover, $\delta(Y) \approx r_{\bt_{i,j}^k}$,
the radius of $\bt_{i,j}^k$,
for each $i,j$, and for every $Y\in \bt_{i,j}^k$.

Let $B(k):= B^k_{1,1}$, $B(k+1):= B^k_{M(N(k)),N(k)}$ be the balls centered
at $X_k$ and $X_{k+1}$ respectively.
By assumption, $|u(X_k)-u(X_{k+1}|>\epsilon$, and by construction,
\eqref{eq4.27} holds with $X= X_k$ and $X=X_{k+1}$.  By these facts, and then
a telescoping argument,
\begin{multline*}
\frac{\epsilon\, k_0}{2}  \leq \sum_{k=1}^{k_0} \left|\fint_{B(k)} \vp -\fint_{B(k+1)} \vp \right|\\[4pt]
\leq\, \sum_{k=1}^{k_0}\sum_{j=0}^{N(k)}\sum_{i=1}^{M(j)}
\left|\fint_{B^k_{i,j}} \vp\, -\,\fint_{\bt^k_{i,j}} \vp\,+\,\fint_{\bt^k_{i,j}}\vp\,-\,\fint_{B^k_{i+1,j}} \vp
\right|\\[4pt]
\lesssim \, \sum_{k=1}^{k_0}\sum_{j=0}^{N(k)}\sum_{i=1}^{M(j)}
\int_{\bt^k_{i,j}} |\nabla \vp(Y)|\,\delta(Y)^{-n}\, dY\,,
\end{multline*}
by Poincare's inequality, since our ambient dimension is $n+1$, and the radius of $\bt^k_{i,j}$
is comparable to $\delta(Y)$ by construction.  Since the balls
$\bt^k_{i,j}$ have bounded overlaps (again by construction), and are all contained in
$\Gamma_{Q_{min}}^{Q_{max}}$, we obtain Claim \ref{claim4.23}
by taking a supremum over all $\left(x,\epsilon,Q_{min},Q_{max}\right)$-admissible 
sequences.
\end{proof}
This concludes the proof of Theorem \ref{t4.18}.
\end{proof}

\section{Remarks on `Qualitative' Fatou Theorems on Rough sets}

In this section, we make the observation that a qualitative Fatou theorem in quite general open sets is a simple consequence of the methods in \cite{ABHM}.

\begin{definition}[Cone Set] Given an open set $\om \subset \ree$ we say $x \in \pom$ is in the cone set of $\om$ if there exists an open truncated cone, $\Gamma_x$, with vertex at $x$ such that $\Gamma_x \subset \om$.
\end{definition}

We define a weakened version of non-tangential limits for open sets that only have qualitative accessibility.

\begin{definition}[Weak Non-tangential Limits] Let  $\om \subset \ree$ be an open set and $x$ in the cone set. We say a function, $u$, has a weak non-tangential limit at $x$ if the limit $\lim_{\substack{y \to x\\ y \in \tilde{\Gamma}_x}} u(y)$ exists, for some 
$\tilde{\Gamma}_x$ an open
truncated cone with vertex at $x$ such that $\tilde{\Gamma}_x \subset \om$. Here the cone $\tilde{\Gamma}_x$ does not need to be $\Gamma_x$ afforded by the fact that $x$ is in the cone set.
\end{definition}

\begin{proposition} Suppose $\om \subset \ree$ is an open set with $\sigma = H^n|_{\pom}$ locally finite and $\sigma(\pom \setminus K) = 0$, where $K$ is the cone set of $\om$. Then every bounded harmonic function in $\om$ has a weak non-tangential limit at $\sigma$-a.e. $x \in \pom$. 
\end{proposition}
\begin{proof}
Let $\om$ be as above and let $u$ be a bounded harmonic function in $\om$. Then following \cite{ABHM}, we may construct $\{\om_i\}_i$ a countable collection of bounded Lipschitz domains such that $\om_i \subset \om$ for all $i$ and $\sigma(\pom \setminus \cup_i \pom_i) = 0$. By \cite{D},
bounded harmonic functions in each $\om_i$ have a (weak) non-tangential limit (relative to $\om_i$) at a.e. $x \in \pom_i$. Thus, $u|_{\om_i}$ has a (weak) non-tangential limit at a.e. $x \in \pom_i$. In particular, $u$ has a (weak) non-tangential limit at a.e. $x \in \pom_i \cap \pom$ (using that the interior cones for $\om_i$ are also cones for $\om$). As $\sigma(\pom \setminus \cup_i \pom_i) =0$ the proposition follows readily.
\end{proof}

\end{document}